\documentclass[article,12pt]{elsarticle}

\usepackage{amssymb}
\usepackage{amsmath}
\usepackage{amssymb}
\usepackage[dvips]{color}

\newtheorem{prop}{Proposition}

\newtheorem{priteo}{Theorem}

\newenvironment{proof}[1][Proof]{\textbf{#1.} }{\ \rule{0.5em}{0.5em}}

\journal{Applied Mathematical Modelling}

\begin{document}

\begin{frontmatter}



\title{Oscillation patterns in tori of modified FHN neurons}


\author{Adrian C. Murza}

\address{Institut de F\'{\i}sica Interdisciplin\`{a}ria i Sistemes Complexos (CSIC-UIB). \\ Campus Universitat de les Illes Balears, Carretera de Valldemossa km. 7.5,\\ 07122 Palma de Mallorca, Espa\~na}

\begin{abstract}
We analyze the dynamics of a network of electrically coupled, modified FitzHugh-Nagumo (FHN) oscillators. The network building-block architecture is a bidimensional squared array shaped as a torus, with unidirectional nearest neighbor coupling in both directions. Linear approximation about the origin of a single torus, reveals that the array is able to oscillate via a Hopf bifurcation, controlled by the interneuronal coupling constants. Group theoretic analysis of the dynamics of one torus leads to discrete rotating waves moving diagonally in the squared array under the influence of the direct product group $\mathbb{Z}_N\times\mathbb{Z}_N\times\mathbb{Z}_2\times\mathbb{S}^1.$ Then, we studied the existence multifrequency patterns of oscillations, in networks formed by two coupled tori. We showed that when acting on the traveling waves, this group leaves them unchanged, while when it acts on the in-phase oscillations, they are shifted in time by $\phi.$ We therefore proved the possibility of a pattern of oscillations in which one torus produces traveling waves of constant phase shift, while the second torus shows synchronous in-phase oscillations, at $N-$ times the frequency shown by the traveling waves.
\end{abstract}


\end{frontmatter}



\section{Introduction}
\label{introd}

The phenomenon of multifrequency oscillatory patterns has been initially described in electrical systems, where one or many variables happened to oscillate with two or three times the frequency of the other oscillators when there was no obvious symmetry reason to do so \cite{AC99} (and references therein). This phenomenon received a solid group theoretic treatment by the works of Golubitsky and Stewart \cite{GS02},\cite{GS04},\cite{GS04b},\cite{GS06},\cite{GS86},\cite{GSB98} as well as Armbruster and Chossat in \cite{AC99} and it has also been confirmed experimentally in \cite{SS07},\cite{LPV03},\cite{LPV07} and \cite{Phys}. Concretely, Palacios et al. in \cite{LPV03},\cite{LPV07} proved the existence of such multifrequency patterns in electrical systems of Duffing oscillators.

Duffing oscillators where the first dynamical systems used to model the brain activity and in particular, to reproduce the electroencephalograms \cite{Z}. However, the first mathematical model for the electrical signaling or firing for individual as well as coupled neurons, developed in direct relationship with the underlying neuronal physiology, was the Hudgkin-Huxley model \cite{HH} and its simplified version, the FHN model \cite{FN1}, \cite{FN2}. Neurophysiological informations on brain processing indicate that neural networks must function in a relatively narrow frequency range \cite{BBK98}. This contrasts with the observation that brain uses a wide frequency spectrum including several oscillation ranges (delta, theta, alpha, beta, gamma), which gives poor efficiencies of signal coding by different frequencies, because of low signal resolution in this frequency range \cite{BBKS00}. It has been proposed that hierarchical feature integration can be accomplished using multifrequency quasiperiodic oscillations \cite{BBK98}. For a two-level system, the authors propose that synchronization of high-frequency oscillations is used to integrate simple features \cite{BBK98}, \cite{martinerie}, whilst for more complex features out-of-phase oscillations of constant phase shift were proposed \cite{coexistence} and \cite{learning}. Multifrequency oscillatory patterns are in fact widely described in neural systems. Many periodic activities are originated by groups of neurons forming central pattern generators \cite{GS02}, \cite{M02}, \cite{KE88}, \cite{KE90},\cite{RTWE98}, \cite{CS94}. In addition, coexistence between in-phase oscillations and traveling waves of constant phase shift have been reported in the nervous system \cite{coexistence} and references therein.

Symmetry is a constantly present feature in nervous system architecture. The already mentioned central pattern generators, visual cortex neuronal circuits in mammals \cite{Cat}, \cite{Cortex}  or plasmodial slime mold \cite{Phys} are just a few examples in which symmetry groups act on the dynamical systems represented by biological oscillators \cite{VisGol}, \cite{GSB98} and \cite{CS94}.
Consider for instance the following system of differential equations, $\displaystyle{\frac{d\mathbf{x}}{dt}=\mathbf{f(x)}}$ where $\mathbf{f}:\mathbb{R}^{n}\rightarrow\mathbb{R}^{n}$ is a smooth vector field. Let $\mathbf{\Gamma}$ be a group that acts in $\mathbb{R}^{n}.$ Let us denote the action of $\mathbf{\Gamma}$ on the space of vector fields on $\mathbb{R}^{n}$ by $\mathbf{\Gamma}$ also. In general we have the result that if the action of a symmetry group $\mathbf{\Gamma}$ on an ordinary differential equation $\displaystyle{\frac{d\mathbf{x}}{dt}=\mathbf{f}(\mathbf{x},\mu)}$, equivariant under the action of the symmetry group, where $\mathbf{x}\in\mathbb{R}^{n},\mu\in\mathbb{R}$ and $\mathbf{f}$ is smooth, then $\gamma \mathbf{f}(\mathbf{x},\mu)=\mathbf{f}(\gamma\mathbf{x},\mu),\forall\gamma\in\mathbf{\Gamma}.$ From the uniqueness of the solutions of initial value problems, it follows that either $x(t)$ and $\gamma\ x(t)$ are disjoint trajectories resulting from a new periodic solution or they differ from each other by a phase shift. This leads to the following conclusion: when a system inherits a symmetry and when it undergoes Hopf bifurcation, resulting periodic solutions also inherit certain symmetry properties \cite{H06}, \cite{GS04b}. We study in this paper, the dynamics of a symmetric network of coupled neurons, that is, a network on which act certain symmetry groups. We focus our analysis on the multifrequency patterns of oscillations which are dependent on the network architecture.

In our approach, the neurons are of FHN type. For a single neuron, the FHN model \cite{M02} is modified according to \cite{SS07}; the variables assumed in this model are the membrane potential $x$ and (a surrogate for) the ionic current $y.$ The state of the neuron is thus specified by a point $(x,y)\in \mathbb{R}^{2}$, and the internal dynamics is
\begin{equation}\label{FHN1}
    \begin{array}{l}
        \dot{x}=ax-x^3-y\\
        \dot{y}=bx-cy
    \end{array}
\end{equation}
where $a, b$ are parameters, and $0<a<1, b>0, c>0$ and $b>a.$ In a two-cell network, the dynamics of one cell is modified by coupling influences from the other cell. This can be imagined, as in \cite{GS06}, by adding an applied current to the $\dot{x}$ equation, this being a function of the state of the other cell. Thus $x_{1} = (v_{1}, w_{1}), x_{2} = (v_{2}, w_{2})$, and $g(u,v)=(ax_{1}-x_{1}^3-y_{1}+k(x_{1}-x_{2}),bx_{1}-cy_{1})$, where $k$ represents coupling strength.

The paper is organized as follows. In section \ref{net} we describe the building block structure of our network: a three-dimensional lattice shaped as a torus, and identify the symmetry group acting on the coupled differential systems located at the nodes of the lattice. In section \ref{linear}, we start the analysis of the network formed by one torus. We analyze the dynamics of the linear approximation of system (\ref{nFHN_2d}) about the origin, by means of the explicit expressions for eigenvalues and eigenvectors, and identify the existence of limit cycles arose from the Hopf bifurcation which depends on the interneuronal couplings. In section \ref{onetorus_theor} we continue to study the dynamics of one torus, this time from a group theoretic angle. We apply the Theorem \ref{equiv} from \cite{GS04}, to identify the two-dimensional isotropy subgroups of the symmetry group acting on the global system \ref{nFHN_2d}, and prove the existence of $N-$ distinct discrete rotating waves which are diagonally translated in the lattice. In section \ref{twocoupled}, we carry our analysis of the dynamics of two coupled tori. We show that certain groups of symmetries can be associated with periodic patterns where an entire array oscillates at different frequencies, and apply this argument to describe the multifrequency patterns of oscillations in two coupled tori.

\section{Network structure}\label{net}

Palacios et al. \cite{PCLR05} as well as Longhini et al. \cite{LPV07} experimentally proved the existence of multifrequencies in unidimensional rings of overdamped Duffing oscillators as well as Van der Pol oscillators \cite{LPV03}. Since our main aim here is studying the patterns of multifrequencies, it appears natural  not only to give a detailed description of our network structure, but also to briefly compare it with the cited network architectures. In \cite{PCLR05} and \cite{LPV03}, it has been analyzed the dynamics of two unidimensional and oriented rings of cells with nearest-neighbor intrarray coupling and all-to-all interarray coupling. In \cite{LPV07} they considered arrays formed by $N$ oriented rings of cells, the network architecture respecting unidirectional coupling in each array and also between adjacent cells in the $N$ arrays.

In particular, the general network presented in \cite{LPV07}, is such that the first row (or column) is not connected to the last row (or column). Their network is constructed from a rectangular array after the top and the bottom edges are identified with each other, such that the final structure respects the cylindrical topology. They used opposite directions of the intraarray couplings.

In this work, we start to build our network, by considering a squared lattice, the nodes of which represent FHN neurons, $N$ in each direction. We first fold the squared array into the shape of a cylinder, followed by closing up the ends of the cylinder. Thus, the building block of our network is a two-dimensional array wrapped in the shape of an empty torus. A similar structure served Miller et al. \cite{MDHM00} to model an array of Van der Pol oscillators as the cartesian product of three rings. Our interarray couplings have same directions; we do not use cross couplings.

As a first approximation we assume that all neurons are identical. Each neuron of the network is coupled to the two forehead neurons (of its four nearest neighbors), one in each direction, and we respect the same direction of interring couplings. For example, neuron located at $(\alpha,\beta)$ in the array, is coupled to neuron $(\alpha+1,\beta)$ along the $\alpha$ direction, and with neuron $(\alpha,\beta+1)$ along the $\beta$ direction (see Figure \ref{fig1}). The dynamics of the neuron at the location $\alpha,\beta$ is described by the variables $x_{\alpha,\beta}, y_{\alpha,\beta}$
\begin{equation}\label{nFHN_2d}
    \begin{array}{l}
        \dot{x}_{\alpha,\beta}=ax_{\alpha,\beta}-x_{\alpha,\beta}^3-y_{\alpha,\beta}+\gamma \kappa(x_{\alpha+1,\beta},x_{\alpha,\beta})+\delta \mu(x_{\alpha,\beta+1},x_{\alpha,\beta})\\
        \dot{y}_{\alpha,\beta}=bx_{\alpha,\beta}-cy_{\alpha,\beta}
    \end{array}
\end{equation}
where 
\begin{equation}\label{coupling}
    \begin{array}{l}
        \kappa(x_{\alpha+1,\beta},x_{\alpha,\beta})=-x_{\alpha+1,\beta}+x_{\alpha,\beta},\\
        \mu(x_{\alpha,\beta+1},x_{\alpha,\beta})=-x_{\alpha,\beta+1}+x_{\alpha,\beta},
    \end{array}
\end{equation}
are the interneuronal intratorus coupling functions,$\gamma$ is the coupling constant in $\alpha$ direction, while $\delta$ is the coupling constant in $\beta$ direction.

\begin{figure}[ht]
\begin{center}
\includegraphics[width=13.8875cm,height=5.4cm]{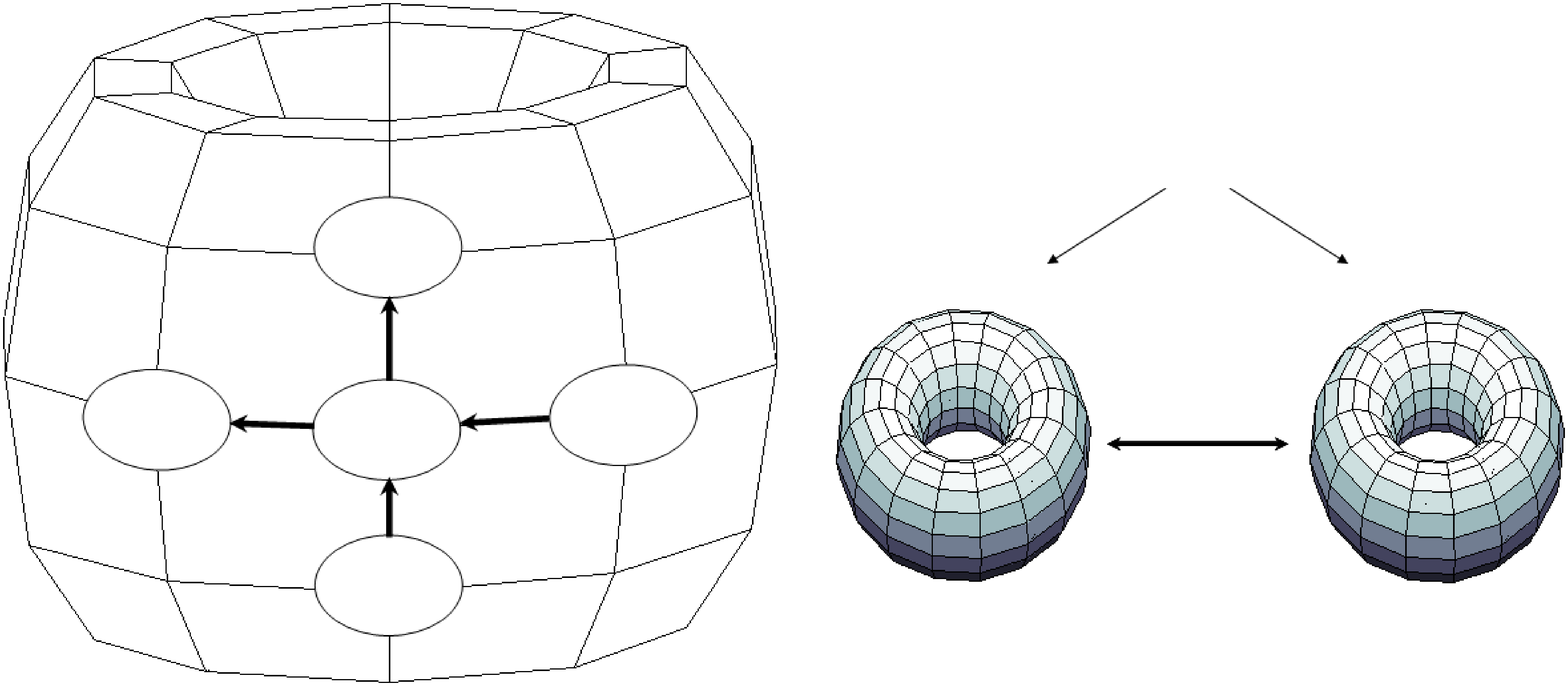}
\caption{\label{fig1} (a) Schematic diagram of the bidimensional array on the torus surface: along the $\alpha$ direction (the coupling constant is $\gamma$), neuron $(\alpha+1,\beta)$ is unidirectionally coupled with neuron $(\alpha,\beta),$ neuron $(\alpha,\beta)$ is coupled with neuron $(\alpha-1,\beta),$ etc. Along the $\beta$ direction (the coupling constant is $\delta$), neuron $(\alpha,\beta+1)$ is unidirectionally coupled with neuron $(\alpha,\beta),$ neuron $(\alpha,\beta)$ is coupled with neuron $(\alpha,\beta-1),$ etc.; (b) Schematic diagram of a $2$-tori network obtained by bidirectional inter-tori coupling (the coupling constant is $\varepsilon$).}
\end{center}
\begin{picture}(0,0)\vspace{0.5cm}
\put(0,120){\large{$\mathbf{(a)}$}}
\put(28,180){\scriptsize{$\mathbf{\alpha+1,\beta}$}}
\put(140,180){\scriptsize{$\mathbf{\alpha-1,\beta}$}}
\put(90,178){\scriptsize{$\mathbf{\alpha,\beta}$}}
\put(85,214){\scriptsize{$\mathbf{\alpha,\beta+1}$}}
\put(85,145){\scriptsize{$\mathbf{\alpha,\beta-1}$}}
\put(68,186){\large{$\mathbf{\gamma}$}}
\put(122,186){\large{$\mathbf{\gamma}$}}
\put(90,159){\large{$\mathbf{\delta}$}}
\put(90,193){\large{$\mathbf{\delta}$}}
\put(195,120){\large{$\mathbf{(b)}$}}
\put(205,210){\scriptsize{torus\hspace{0.1cm}\#1}}
\put(347,210){\scriptsize{torus\hspace{0.1cm}\#2}}
\put(269,167){\scriptsize{bidirectional}}
\put(274,157){\scriptsize{inter-tori}}
\put(276,147){\scriptsize{coupling}}
\put(288,180){\large{$\mathbf{\mathbf{\varepsilon}}$}}
\put(250,255){\scriptsize{unidirectional intra-ring}}
\put(247,245){\scriptsize{nearest-neighbor coupling}}
\put(263,235){\scriptsize{on torus surface}}
\end{picture}\vspace{-0.7cm}
\label{first_figure}
\end{figure}

In the following, we will try to identify and describe the symmetry group of one torus.
Let us denote by $\mathbb{Z}_{N}$, the group of cyclic permutations of $N$ neurons in each direction on the torus surface. In this torus-like neuronal network with oriented nearest neighbor coupling, the symmetry group is not $\mathbb{Z}_{N}\times\mathbb{Z}_{N}$ but $\mathbb{Z}_{N}\times\mathbb{Z}_{N}\times\mathbb{Z}_{2}$, where the additional symmetry is introduced by the property $(x,y)\mapsto(-x,-y).$ More specifically, the action of this group on $\mathbb{R}^{N^2}\oplus\mathbb{R}^{N^2}$, with coordinates $(x,y)$ is:
\begin{equation*}\label{torus_dihedral}
    \begin{array}{l}
        \sigma (x_{1,\beta},\ldots,x_{N,\beta};y_{1,\beta},\ldots,y_{N,\beta})=(x_{2,\beta},\ldots,x_{N,\beta},x_{1,\beta};y_{2,\beta},\ldots,y_{N,\beta},y_{1,\beta})\\
        \rho (x_{\alpha,1},\ldots,x_{\alpha,N};y_{\alpha,1},\ldots,y_{\alpha,N})=(x_{\alpha,2},\ldots,x_{\alpha,N},x_{\alpha,1};y_{\alpha,2},\ldots,y_{\alpha,N},y_{\alpha,1})\\
        \varpi(x_{1,\beta},\ldots,x_{N,\beta};y_{1,\beta},\ldots,y_{N,\beta};x_{\alpha,1},\ldots,x_{\alpha,N};y_{\alpha,1},\ldots,y_{\alpha,N})=\\
        =(-x_{1,\beta},\ldots,-x_{N,\beta};-y_{1,\beta},\ldots,-y_{N,\beta};-x_{\alpha,1},\ldots,-x_{\alpha,N};-y_{\alpha,1},\ldots,-y_{\alpha,N})
    \end{array}
\end{equation*}
where $\mathbb{Z}_{N}^{\alpha}=\langle\sigma\rangle,$ $\mathbb{Z}_{N}^{\beta}=\langle\rho\rangle,$ and $\mathbb{Z}_{2}=\langle\varpi\rangle.$ Let $\mathbf{\Gamma}=\mathbb{Z}_{N}^{\alpha}\times\mathbb{Z}_{N}^{\beta}\times\mathbb{Z}_{2}.$

In a network with cylindrical topology, coupling functions $\kappa\left(x_{\alpha+1,\beta},x_{\alpha,\beta}\right)$ and $\mu\left(x_{\alpha,\beta+1},x_{\alpha,\beta}\right)$ fail to be $\mathbf{\Gamma}$ equivariant for the reason that neurons in the first row (or column) are not treated as neighbors of neurons in the last row (or column) and vice-versa. However; if the two free ends of the cylinder are connected to eachother, and thereby declaring first and last rows (or columns) as adjacent, then the resulting $\kappa(x_{\alpha+1,\beta},x_{\alpha,\beta})$ and $\mu(x_{\alpha,\beta+1},x_{\alpha,\beta})$ turn out to be $\mathbf{\Gamma}$ equivariant.

In our attempt to study the multifrequency patterns in two coupled tori, we would like to analyze first, the dynamics of a single torus. The main question we try to answer to in the next section is the following: how does the coupling functions/parameters affect the network dynamics? Linearization about the origin leads us to the following.

\begin{figure}[ht]
\centering
\begin{center}
\includegraphics[width=11.039cm,height=3.3766cm]{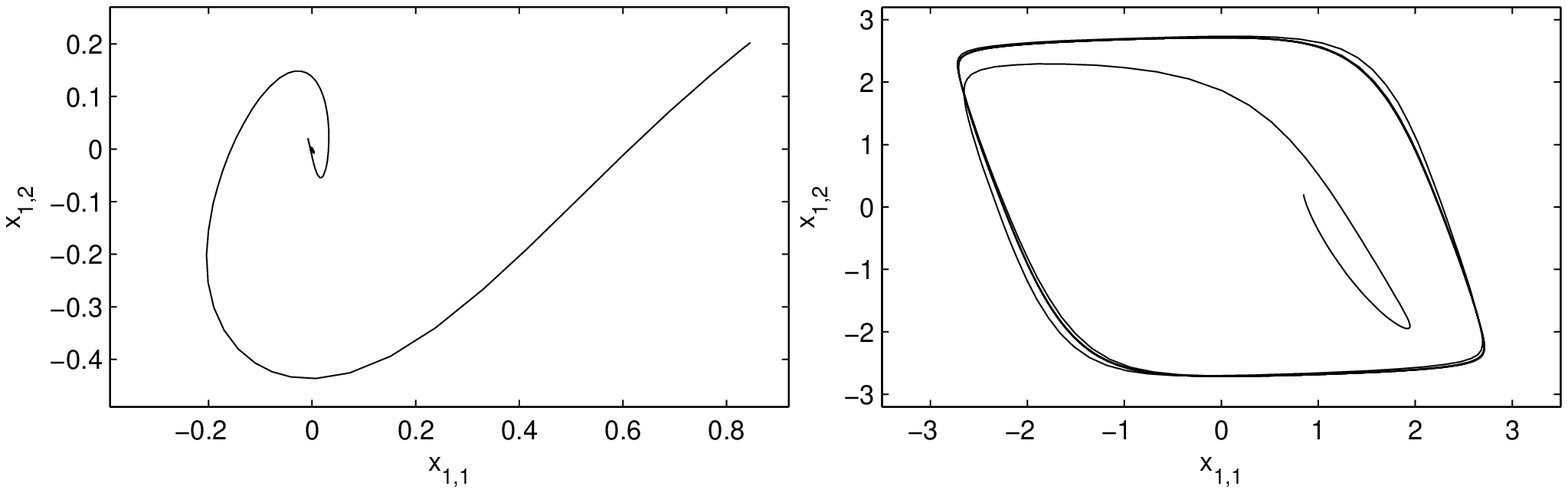}
\caption{\label{fig2_un_toro} Plot of $x_{1,1}$ versus $x_{1,2}$ from the integration of system (\ref{nFHN_2d}) corresponding to a torus of $3\times 3$ neurons: (a) low coupling case, $\gamma=0.1$ and $\delta=0.1.$ The trajectory describes an open curve from the initial conditions to the stable equilibrium point; (b) high coupling case, $\gamma=2$ and $\delta=2.$ The trajectory approaches a closed curve indicating the closeness of a periodic solution. The rest of the parameters and initial conditions for all variables in both cases are $a=0.01$, $b=c=0.9$, $x^{0}=(0.8462,0.2026,0.8381,0.6813,0.8318, 0.7095,0.3046,0.1934,0.3028), y^{0}=(0.5252,0.6721,0.0196,0.3795,0.5028,0.4289,0.1897,0.6822,0.5417).$}
\begin{picture}(0,0)\vspace{0.5cm}
\put(-145,120){\large{$\mathbf{(a)}$}}
\put(0,120){\large{$\mathbf{(b)}$}}
\put(150,180){\includegraphics[width=1.6825cm,height=1.2225cm]{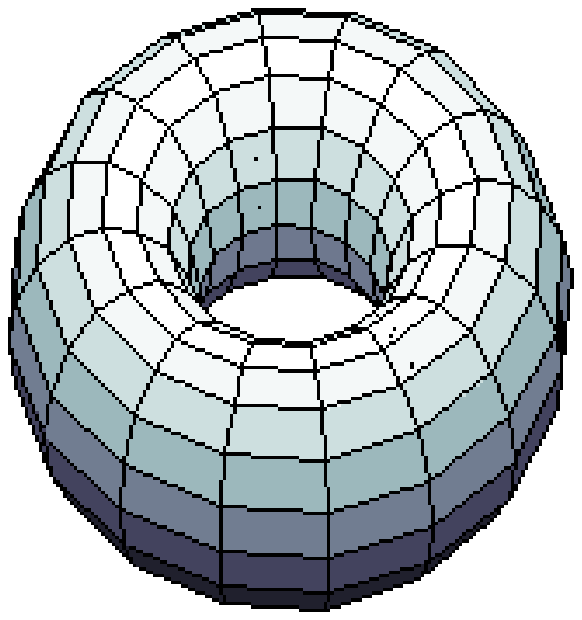}}
\put(165,170){\scriptsize{single}}
\put(166,160){\scriptsize{torus}}
\put(162,150){\scriptsize{array of}}
\put(167,140){\scriptsize{$3\times 3$}}
\put(163,130){\scriptsize{neurons}}
\end{picture}\vspace{-1.3cm}
\end{center}
\label{second_figure}
\end{figure}

\section{Linear analysis}\label{linear}

To analyze bifurcation of the system (\ref{nFHN_2d}) we need to understand the eigenvalues and eigenvectors of the linear approximation about the origin. Origin is a singular point of (\ref{nFHN_2d}) and linearization leads to the $N\times N$ block circulant matrix $M$

\begin{equation*}\label{matrixM}
    M =
        \begin{bmatrix}
            A&B&\mathbf{0}&\ldots&\mathbf{0}&\mathbf{0}\\
            \mathbf{0}&A&B&\mathbf{0}&\ldots&\mathbf{0}\\
            \mathbf{0}&\mathbf{0}&A&B&\ldots&\mathbf{0}\\
            \vdots&\vdots&\vdots&\vdots&\vdots&\vdots&\\
            B&\mathbf{0}&\ldots&\mathbf{0}&\mathbf{0}&A\\
        \end{bmatrix}
\end{equation*}
where $A$ is an $N\times N$ block circulant matrix and $B$ is an $N\times N$ block diagonal matrix.

\begin{equation*}\label{matricesAB}
    A =
        \begin{bmatrix}
            D&E&\mathbf{0}&\ldots&\mathbf{0}&\mathbf{0}\\
            \mathbf{0}&D&E&\mathbf{0}&\ldots&\mathbf{0}\\
            \mathbf{0}&\mathbf{0}&D&E&\ldots&\mathbf{0}\\
            \vdots&\vdots&\vdots&\vdots&\vdots&\vdots&\\
            E&\mathbf{0}&\ldots&\mathbf{0}&\mathbf{0}&D\\
        \end{bmatrix}
    B =
        \begin{bmatrix}
            F&\mathbf{0}&\mathbf{0}&\ldots&\mathbf{0}&\mathbf{0}\\
            \mathbf{0}&F&\mathbf{0}&\mathbf{0}&\ldots&\mathbf{0}\\
            \mathbf{0}&\mathbf{0}&F&\mathbf{0}&\ldots&\mathbf{0}\\
            \vdots&\vdots&\vdots&\vdots&\vdots&\vdots&\\
            \mathbf{0}&\mathbf{0}&\ldots&\mathbf{0}&\mathbf{0}&F\\
            \end{bmatrix}
\end{equation*}
Matrices $D$, $E$ and $F$ are given by
\begin{equation*}\label{matricesCDE}
    D =
        \begin{bmatrix}
            d&-1\\
            b&-c\\
        \end{bmatrix}
    E =
        \begin{bmatrix}
            -\gamma&0\\
            0&0\\
        \end{bmatrix}
    F =
        \begin{bmatrix}
            -\delta&0\\
            0&0\\
        \end{bmatrix}
\end{equation*}
where $d=a+\gamma+\delta.$ Let us define $G_{(r,s)}=D+\zeta_r E+\zeta_s F$, where $\zeta_r=e^\frac{2\pi ir}{N}$ and $\zeta_r=e^\frac{2\pi is}{N}$ are the $r^{th}$ and $s^{th}$ roots of unity, respectively, with $0\leq r,s\leq N-1.$ To find the eigenvectors of $M$, we first start with the eigenvectors of $G_{(r,s)},$ which are $\displaystyle{{v_{\left(r,s\right)}}_{1,2}=\left[1,a+\gamma\left(1-\zeta_r\right)+\delta\left(1-\zeta_s\right)-{\lambda_{\left(r,s\right)}}_{1,2}\right]^T},$ where ${\lambda_{\left(r,s\right)}}_{1,2}$ are the eigenvalues corresponding to each pair $(r,s).$ Then define $\displaystyle{{w_{\left(r,s\right)}}_{1,2}=\left[{v_{\left(r,s\right)}}_{1,2},\ldots,{v_{\left(r,s\right)}}_{1,2}\right]^T}.$

\begin{figure}[ht]
\centering
\begin{center}
\includegraphics[width=6.77966cm,height=5.08474cm]{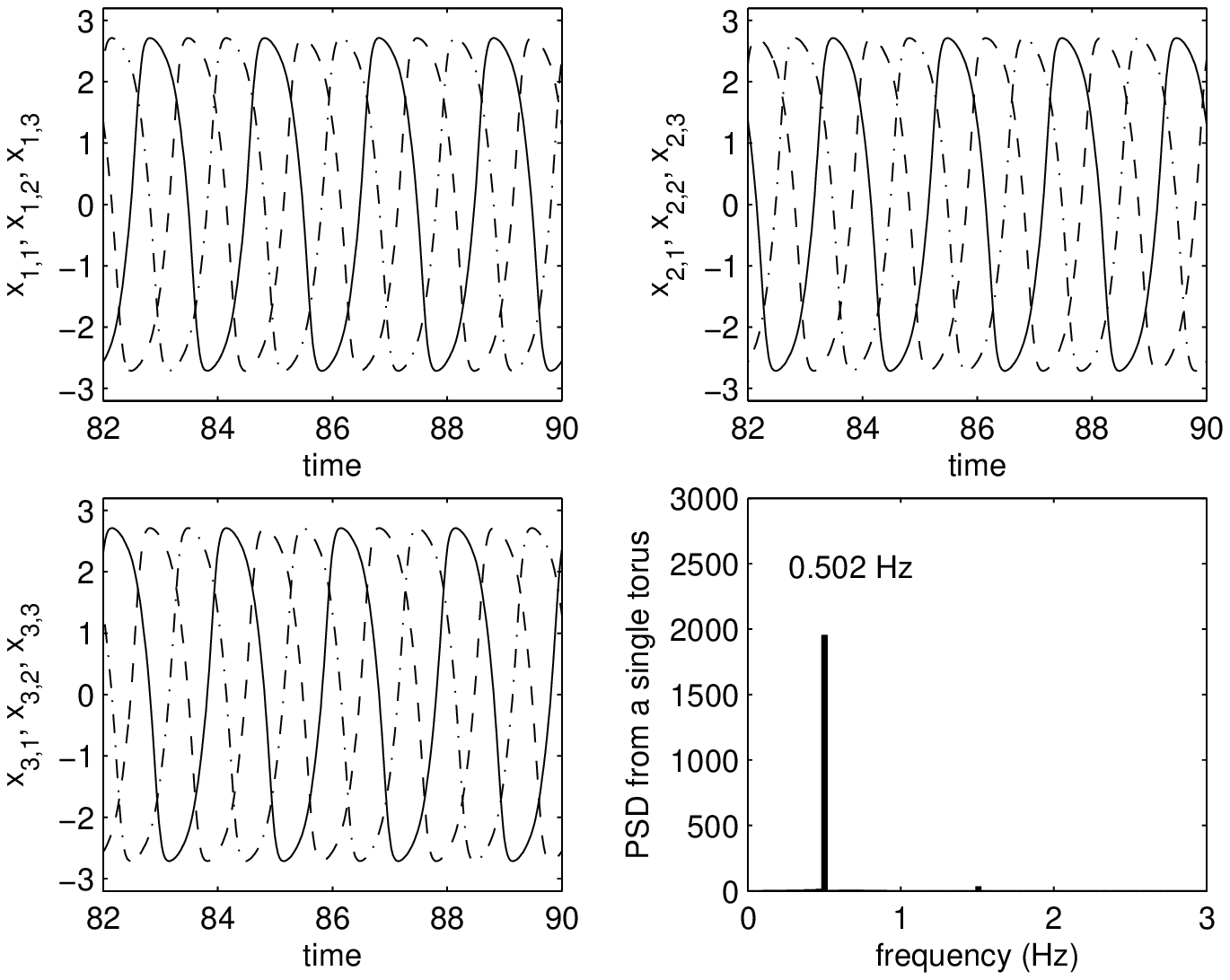}
\caption{\label{fig3_un_toro} Time series of variables a) $x_{1,1}$ (continuous line), $x_{1,2}$ (dashed line) , $x_{1,3}$ (dotted and dashed line); b) $x_{2,1}$ (continuous line), $x_{2,2}$ (dashed line) , $x_{2,3}$ (dotted and dashed line) and c) $x_{3,1}$ (continuous line), $x_{3,2}$ (dashed line) , $x_{3,3}$ (dotted and dashed line) from the integration of system (\ref{nFHN_2d}) corresponding to a torus of $3\times 3$ neurons. Coupling constants are, $\gamma=2$ and $\delta=2$, while the rest of the parameters and initial conditions are the same as in Figure \ref{fig2_un_toro}. There is a constant phase shift between the variables belonging to the same row. Power Spectrum Density indicating that all variables oscillate with the same frequency.}
\begin{picture}(0,0)\vspace{0.5cm}
\put(2,120){\large{$\mathbf{(d)}$}}
\put(-93,120){\large{$\mathbf{(c)}$}}
\put(2,195){\large{$\mathbf{(b)}$}}
\put(-93,195){\large{$\mathbf{(a)}$}}
\put(110,190){\includegraphics[width=1.6825cm,height=1.2225cm]{1toro.eps}}
\put(115,180){\scriptsize{single torus}}
\put(94,170){\scriptsize{array of $3\times 3$ neurons}}
\end{picture}\vspace{-1cm}
\end{center}
\label{third_figure}
\end{figure}

To find the eigenvectors of $M,$ we proceed as in \cite{T97} so that for each pair $(r,s),$ we obtain the corresponding pair of eigenvectors
\begin{equation*}\label{eigvects}
    \begin{array}{l}
        \displaystyle{{\nu_{\left(r,s\right)}}_{1,2}=\frac{1}{N}\left[{w_{\left(r,s\right)}}_{1,2},\zeta^{k}{w_{\left(r,s\right)}}_{1,2},\zeta^{2k}{w_{\left(r,s\right)}}_{1,2},\ldots,\zeta^{\left(N-1\right)k}{w_{\left(r,s\right)}}_{1,2}\right]^T}
    \end{array}
\end{equation*}
where $0\leq k\leq N-1.$

It is straightforward to derive the explicit expression of the $2N^2$ eigenvalues of $M;$ there is a pair of complex conjugate eigenvalues of the form
\begin{equation}\label{eigvals}
    \begin{array}{l}
        \displaystyle{{\lambda_{\left(r,s\right)}}_{1,2}=\frac{1}{2}\left[-c+a+\gamma(1-\zeta_r)+\delta(1-\zeta_s)\right]}\\
        \hspace{1.7cm}\displaystyle{\pm\frac{1}{2}\sqrt{\left[c+a+\gamma(1-\zeta_r)+\delta(1-\zeta_s)\right]^2-4b}}.
    \end{array}
\end{equation}
corresponding to each pair $(r,s).$

If $\lambda_{1,2}$ correspond to the mode $(r,s)=(0,0)$ then expression (\ref{eigvals}) reduces to the eigenvalues of system (\ref{FHN1}) of uncoupled neurons, linearized about the origin, \begin{equation}\label{eigvalsFHN1}
    \begin{array}{l}
        \displaystyle{\lambda_{1,2}=\frac{-c+a\pm\sqrt{(c+a)^2-4b}}{2}}.
    \end{array}
\end{equation}
Since $\frac{\partial Re\lambda}{\partial a}\neq0$ and $\frac{\partial Re\lambda}{\partial c}\neq0$ it follows that if taking $a$ or $c$ as bifurcation parameter, the pair of complex conjugate eigenvalues (\ref{eigvalsFHN1}) will cross the imaginary axis of the complex plane at $c=a$ whenever $c^2<b,$ and system (\ref{nFHN_2d}) undergoes a Hopf bifurcation within the mode $(r,s)=(0,0).$ Notice that the Hopf bifurcation does not depend on the interneuronal coupling constants.

We will now analyze the conditions for Hopf bifurcation of system (\ref{nFHN_2d}), when not both $r$ and $s$ are zero simultaneously. Let $\mathcal{RS}$ be the set of combinations $(r,s);$ moreover, the set $\mathcal{NZ}=\mathcal{RS}\setminus(0,0),$ defines the combinations set of $(r,s),$ where not both $r$ and $s$ are zero simultaneously. Let us define
        \begin{align}\label{a_1}
            a_1 &=\left(a+c\right)^2+2\left(a+c\right)\left(\gamma-\gamma\cos{\frac{2\pi r}{N}}+\delta-\delta\cos{\frac{2\pi s}{N}}\right)
                \nonumber\\
                &\quad+\gamma^2\left(1-2\cos{\frac{2\pi r}{N}}+\cos{\frac{4\pi r}{N}}\right)
                \nonumber\\
                &\quad+\delta^2\left(1-2\cos{\frac{2\pi s}{N}}+\cos{\frac{4\pi s}{N}}\right)\\
                &\quad+2\gamma\delta\left[1-\cos{\frac{2\pi r}{N}}-\cos{\frac{2\pi s}{N}}+\cos{\frac{2\pi\left(r+s\right)}{N}}\right]
                \nonumber\\
                &\quad-4b,
                \nonumber
        \end{align}
and
        \begin{align}\label{b_1}
            b_1 &=2\left(a+c\right)\left(-\gamma\sin{\frac{2\pi r}{N}}-\delta\sin{\frac{2\pi s}{N}}\right)
                \nonumber\\
                &\quad+\gamma^2\left(-2\sin{\frac{2\pi r}{N}}+\sin{\frac{4\pi r}{N}}\right)+\delta^2\left(-2\sin{\frac{2\pi s}{N}}+\sin{\frac{4\pi s}{N}}\right)\\
                &\quad+2\gamma\delta\left[-\sin{\frac{2\pi r}{N}}-\sin{\frac{2\pi s}{N}}+\sin{\frac{2\pi\left(r+s\right)}{N}}\right].
                \nonumber
        \end{align}

In the following Proposition we summarize the conditions for Hopf bifurcation and stability of system (\ref{nFHN_2d}), when $(r,s)\in\mathcal{NZ}.$
\begin{prop}\label{proposition}
Consider the linearization of system (\ref{nFHN_2d}) around $(0,\ldots,0),$ and assume the square root of eigenvalues in equation (\ref{eigvals}) is a complex number of the form $\sqrt{g+ih},$ where in general, $h\neq0.$ Then, if $(r,s)\in\mathcal{NZ},$
    \begin{itemize}
        \item [(i)] Not both real parts $\mathrm{Re}{\lambda_{\left(r,s\right)}}_{1,2},$ of the eigenvalues corresponding to the pair $(r,s),$  can be zero simultaneously.
        \item [(ii)] If taking as bifurcation parameter $\gamma$ or $\delta,$ and $\frac{\partial \mathrm{Re}{\lambda_{\left(r,s\right)}}}{\partial\gamma}\neq0$ or $\frac{\partial \mathrm{Re}{\lambda_{\left(r,s\right)}}}{\partial\delta}\neq0$ then a Hopf bifurcation occurs when $-c+a+\gamma\left(1-\cos{\frac{2\pi r}{N}}\right)+\delta\left(1-\cos{\frac{2\pi s}{N}}\right)=\left\{\pm\frac{1}{\sqrt{2}}\sqrt{\sqrt{a_1^2+b_1^2}+a_1}\right\},$
        \item [(iii)] Origin is a stable equilibrium of system (\ref{nFHN_2d}) if $-c+a+\gamma\left(1-\cos{\frac{2\pi r}{N}}\right)+\delta\left(1-\cos{\frac{2\pi s}{N}}\right)<\mathrm{min}\left\{\pm\frac{1}{\sqrt{2}}\sqrt{\sqrt{a_1^2+b_1^2}+a_1}\right\},$ unstable otherwise.
    \end{itemize}
\end{prop}
\begin{proof}
    \begin{itemize}
        \item [(i)] Let us suppose the contrary, i.e. that we have $\displaystyle{\mathrm{Re}{\lambda_{\left(r,s\right)}}_{1}=\mathrm{Re}{\lambda_{\left(r,s\right)}}_{2}}=0.$ We will prove that this leads to contradiction. In order to evaluate the global real part as well as the global imaginary part of eigenvalues of the extended system (\ref{nFHN_2d}), we need to rewrite equation (\ref{eigvals}) by getting rid of the square root.

For this purpose, we infer a well known result from elementary algebra; we have that if $g$ and $h$ are real $(h\neq0),$ then $\sqrt{g+ih}=a_2+b_2i,$ where $a_2$ and $b_2$ are real and given by
\begin{equation}\label{a_2}
    \begin{array}{l}
        \displaystyle{a_2=\frac{1}{\sqrt{2}}\sqrt{\sqrt{g^2+h^2}+g}},
    \end{array}
\end{equation}
and
\begin{equation}\label{b_2}
    \begin{array}{l}
        \displaystyle{b_2=\frac{\mathrm{sgn}\left(h\right)}{\sqrt{2}}\sqrt{\sqrt{g^2+h^2}-g}}.
    \end{array}
\end{equation}
Therefore, if
$\sqrt{g+ih}=\sqrt{\left[c+a+\gamma(1-\zeta_r)+\delta(1-\zeta_s)\right]^2-4b},$ then $g=a_1,$ where $a_1$ is given in equation (\ref{a_1}), and $h=b_1$ is given in equation (\ref{b_1}).

From equations (\ref{eigvals}), (\ref{a_1}), (\ref{b_1}), (\ref{a_2}) and (\ref{b_2}), we have that the global real and imaginary parts of eigenvalues (\ref{eigvals}) are given by

\begin{equation}\label{real}
    \begin{array}{l}
        \displaystyle{\mathrm{Re}{\lambda_{\left(r,s\right)}}_{1,2}=\frac{\left[-c+a+\gamma\left(1-\cos{\frac{2\pi r}{N}}\right)+\delta\left(1-\cos{\frac{2\pi s}{N}}\right)\pm{a_2}\right]}{2}}
    \end{array}
\end{equation}
and
\begin{equation*}\label{imag}
    \begin{array}{l}
        \displaystyle{\mathrm{Im}{\lambda_{\left(r,s\right)}}_{1,2}=\frac{i}{2}\left(-\gamma\sin{\frac{2\pi r}{N}}-\delta\sin{\frac{2\pi s}{N}}\pm{b_2}\right)},
    \end{array}
\end{equation*}
respectively, where $a_2$ and $b_2$ are functions of $a_1$ and $b_1$ given by equations (\ref{a_2}) and (\ref{b_2}), respectively. It follows that we can rewrite equation (\ref{eigvals}) as
\begin{equation*}\label{eigvals_rewritten}
    \begin{array}{l}
        \displaystyle{{\lambda_{\left(r,s\right)}}_{1,2}=\mathrm{Re}{\lambda_{\left(r,s\right)}}_{1,2}+i\mathrm{Im}{\lambda_{\left(r,s\right)}}_{1,2}}.
    \end{array}
\end{equation*}
Since we supposed that $\displaystyle{\mathrm{Re}{\lambda_{\left(r,s\right)}}_{1}=\mathrm{Re}{\lambda_{\left(r,s\right)}}_{2}}=0,$  from equation (\ref{real}) we have $a_2=0,$ which means (from equation (\ref{a_2})) that $b_1=0,$ which contradicts the initial assumption that $h\neq0.$
        \item [(ii)] If applying condition $\displaystyle{\mathrm{Re}{\lambda_{\left(r,s\right)}}_{1}=0}$ or $\displaystyle{\mathrm{Re}{\lambda_{\left(r,s\right)}}_{2}=0},$ to equation (\ref{real}), we obtain the desired inequality.
        \item [(iii)] If we put the condition  $\displaystyle{\mathrm{Re}{\lambda_{\left(r,s\right)}}_{1,2}<0}$ in equation (\ref{real}), we obtain $-c+a+\gamma\left(1-\cos{\frac{2\pi r}{N}}\right)+\delta\left(1-\cos{\frac{2\pi
        s}{N}}\right)<\mathrm{min}\left\{\pm\frac{1}{\sqrt{2}}\sqrt{\sqrt{a_1^2+b_1^2}+a_1}\right\}.$
    \end{itemize}
\end{proof}

The linearized flow in $\mathbb{R}^{2N^2}$ possesses periodic solutions if $\mathrm{Im}{\lambda_{\left(r,s\right)}}_{1,2}\neq0.$ Suppose that $\frac{\partial \mathrm{Re}{\lambda_{\left(r,s\right)}}}{\partial\gamma}\neq0$ or $\frac{\partial \mathrm{Re}{\lambda_{\left(r,s\right)}}}{\partial\delta}\neq0$ in (\ref{real}). Then, Proposition \ref{proposition} tells that by varying the coupling parameters one can qualitatively modify the periodic behavior of the coupled system (\ref{nFHN_2d}), with respect to uncoupled neurons (\ref{FHN1}). A remarkable property of the coupled system resides in the fact that self sustained oscillatory behavior of the global system can appear even if individual neurons taken separately act as damped oscillators, as seen in Figure \ref{fig2_un_toro}. This behavior is determined by the value of the coupling constants $\gamma,\delta,$ as well as $N.$
Let $a_3=\frac{1}{2}\left[\gamma\left(1-\cos{\frac{2\pi r}{N}}\right)+\delta\left(1-\cos{\frac{2\pi s}{N}}\right)\pm{a_2}\right].$ The mechanism of this coupling-induced difference of the system behavior can be better seen if directly comparing the $\mathrm{Re}{\lambda_{\left(r,s\right)}}_{1,2}$ from equation (\ref{real}) of the coupled system (\ref{nFHN_2d}) with ${Re\lambda}_{1,2}$ from equation (\ref{eigvalsFHN1}) of the one-neuron-system (\ref{FHN1})
\begin{equation}\label{real_comparison}
    \begin{array}{l}
        \displaystyle{\mathrm{Re}{\lambda_{\left(r,s\right)}}_{1,2}={Re\lambda}_{1,2}+a_3}.
    \end{array}
\end{equation}
From equation (\ref{real_comparison}), we have that if $\frac{\partial a_3}{\partial \gamma}>0$ or $\frac{\partial a_3}{\partial \delta}>0,$ system loses stability by increasing $\gamma, \delta,$ by approaching the Hopf bifurcation point.

When $N$ is even, however, a condition for the Hopf bifurcation of the discrete rotating waves, is taking into account not only nearest-neighbor but also next-nearest neighbor couplings \cite{GS86}, \cite{LPV03}.

An interesting question is determining the type (supercritical or subcritical) of the Hopf bifurcation obtained by linearization about origin, described in this section.
The first step in this analysis consists in studying the stability of the limit cycle arose from the Hopf bifurcation in the three-parameter family of modified FHN systems \eqref{FHN1}. Hopf bifurcation of the classical FHN system \cite{M02} has been studied by many authors (see for example \cite{kostova}), but these results are not applicable to our equation \eqref{FHN1}, which is a modified FHN system. In the Appendix we apply a stability criterium due to Guckenheimer and Holmes \cite{guck}, to prove that the Hopf bifurcation of system \eqref{FHN1} obtained by linearization about origin, is supercritical. The stability of the limit cycle experimented by the coupled FHN systems within one torus (represented by equation \eqref{nFHN_2d}), as well as coupled tori with different intertori architectures (in the case of two tori, represented by equation \eqref{nFHN_2d_2tori}), is an extense analytical study involving reduction to the central manifold, as well as a bifurcation numerical analysis and it will be presented in a forthcoming paper.

\begin{figure}[ht]
\begin{center}
\includegraphics[width=6.77966cm,height=5.08474cm]{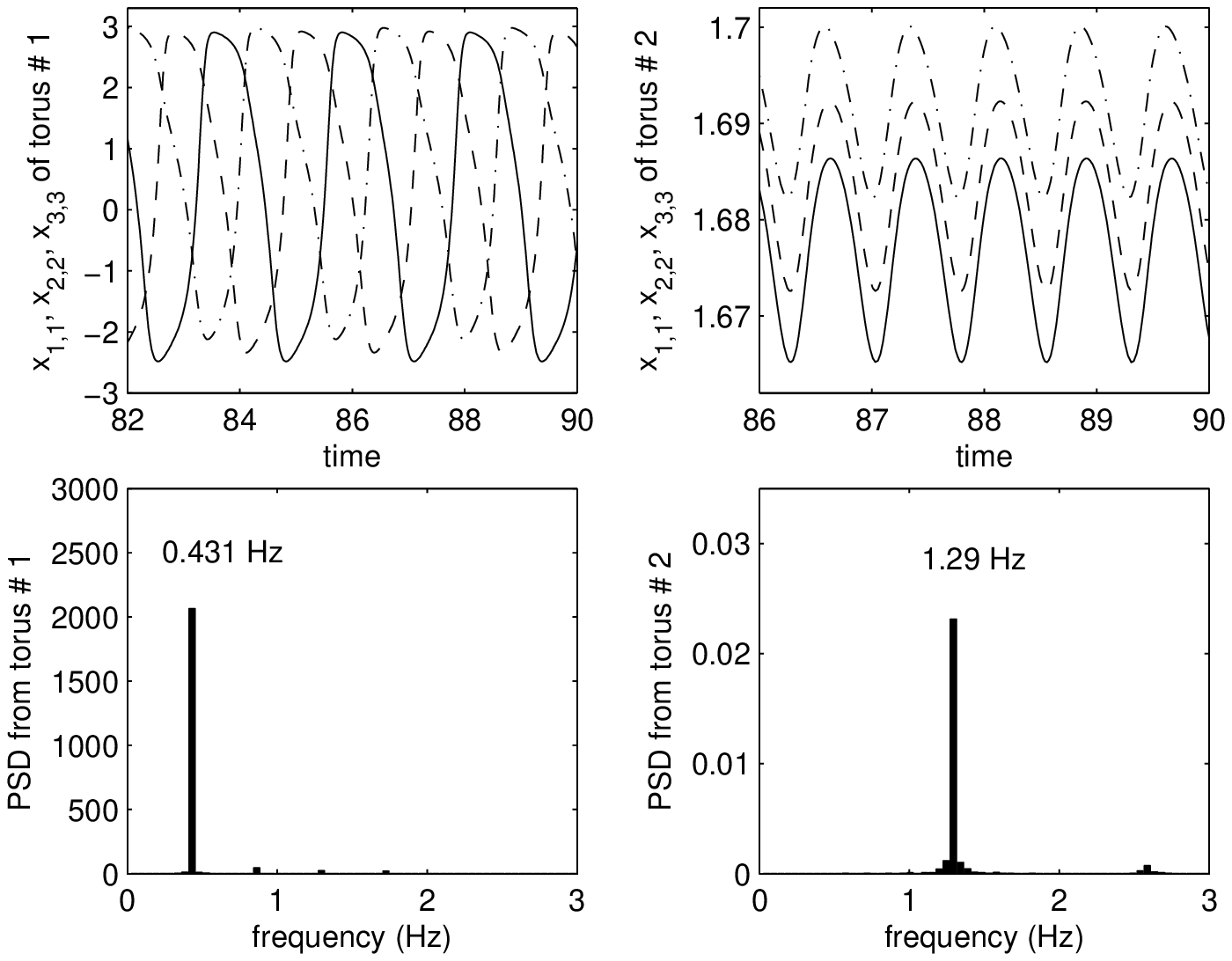}
\includegraphics[width=6.77966cm,height=5.08474cm]{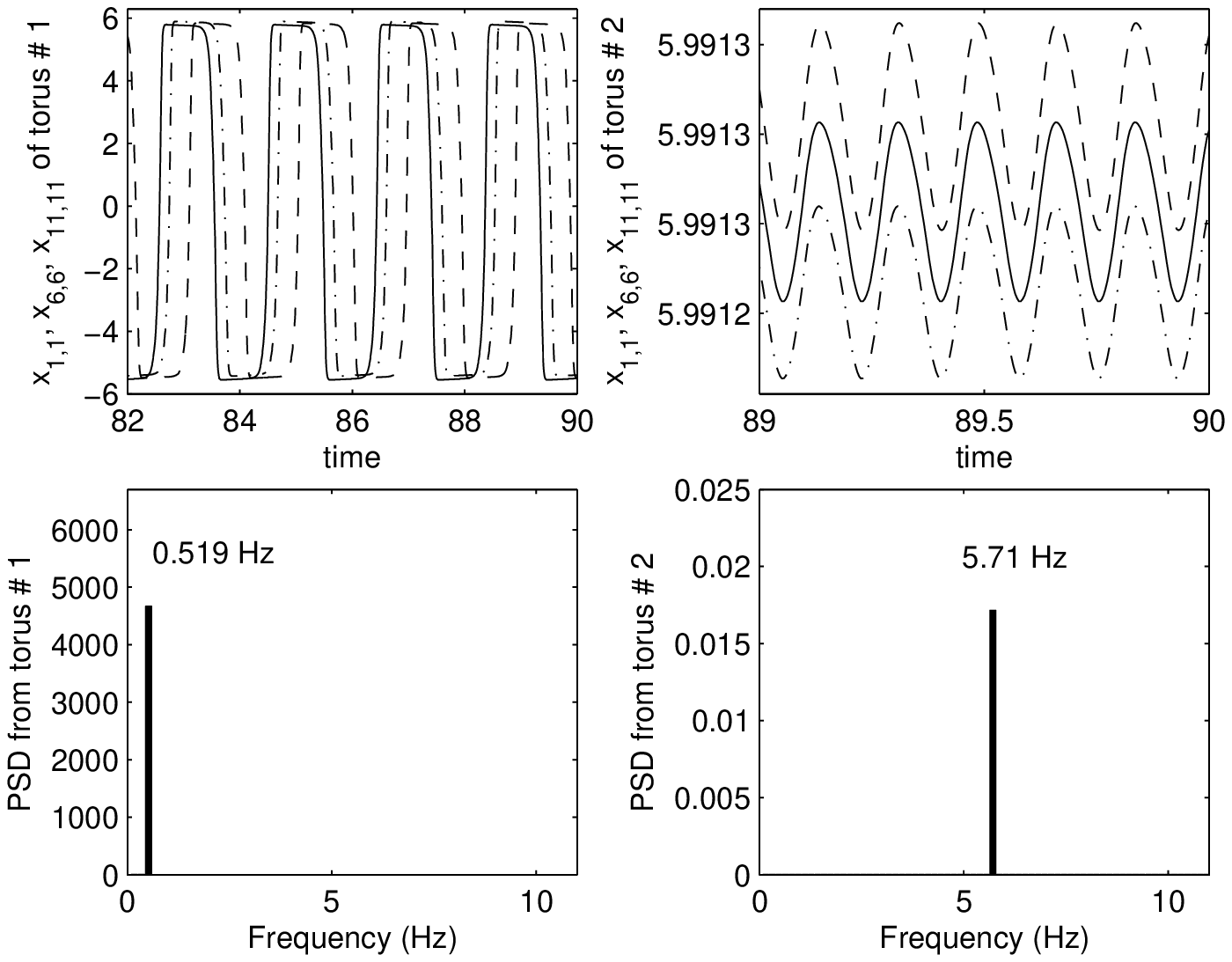}
\end{center}
\begin{picture}(0,0)
\put(298,102){\large{$\mathbf{(f)}$}}
\put(200,102){\large{$\mathbf{(e)}$}}
\put(98,102){\large{$\mathbf{(b)}$}}
\put(2,102){\large{$\mathbf{(a)}$}}
\put(298,30){\large{$\mathbf{(h)}$}}
\put(200,30){\large{$\mathbf{(g)}$}}
\put(98,30){\large{$\mathbf{(d)}$}}
\put(2,30){\large{$\mathbf{(c)}$}}
\put(63,-7){\includegraphics[width=2.626cm,height=0.978cm]{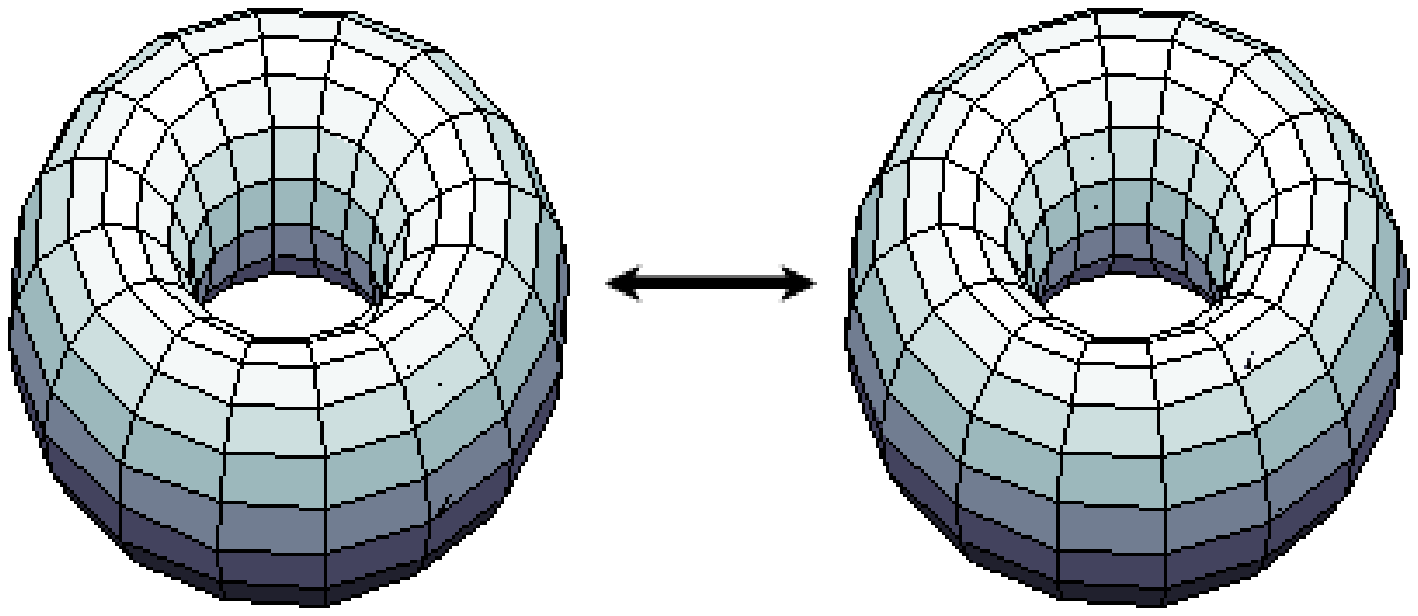}}
\put(268,-7){\includegraphics[width=2.626cm,height=0.978cm]{2tori_fig_4.eps}}
\put(43,185){\tiny{two tori of}}
\put(37,175){\tiny{$3\times 3$ neurons}}
\put(135,185){\tiny{two tori of}}
\put(129,175){\tiny{$3\times 3$ neurons}}
\put(238,185){\tiny{two tori of}}
\put(229,175){\tiny{$11\times 11$ neurons}}
\put(332,185){\tiny{two tori of}}
\put(323,175){\tiny{$11\times 11$ neurons}}
\put(43,65){\tiny{two tori of}}
\put(37,55){\tiny{$3\times 3$ neurons}}
\put(152,75){\tiny{two tori of}}
\put(158,65){\tiny{$3\times 3$}}
\put(154,55){\tiny{neurons}}
\put(238,65){\tiny{two tori of}}
\put(229,55){\tiny{$11\times 11$ neurons}}
\put(317,75){\tiny{two tori of}}
\put(320,65){\tiny{$11\times 11$}}
\put(320,55){\tiny{neurons}}
\put(41,-20){\scriptsize{two tori of $3\times 3$ neurons each}}
\put(242,-20){\scriptsize{two tori of $11\times 11$ neurons each}}
\end{picture}
\vspace{0.7cm}
\caption{Time series of variables $x_{1,1}$ (continuous line), $x_{2,2}$ (dashed line) , $x_{3,3}$ (dotted and dashed line) from the torus $\#1$ (a) and torus $\#2$ (b), from the integration of system (\ref{nFHN_2d}) corresponding to two coupled tori, each of $3\times 3$ neurons. Variables are phase shifted in torus $\#1$, while they are synchronized in torus $\#2$. Coupling constants are, $\gamma=2$ and $\delta=2,$ $\varepsilon=0.5$, while the rest of the parameters and initial conditions are the same as in Figure \ref{fig2_un_toro}. Power Spectrum Density showing that neurons of torus $\#2$ (d) oscillate at $1.29=3\times 0.431$ Hz i.e. with 3--times the frequency of neurons from torus $\#1$ (c). Time series of variables $x_{1,1}$ (continuous line), $x_{6,6}$ (dashed line) , $x_{11,11}$ (dotted and dashed line) from the torus $\#1$ (e) and torus $\#2$ (f), from the integration of system (\ref{nFHN_2d}) corresponding to two coupled tori, each of $11\times 11$ neurons.
Neurons are phase shifted in torus $\#1$, while they are synchronized in torus $\#2$. Coupling constants are, $\gamma=8$ and $\delta=8$, $\varepsilon=0.5$, while the rest of the parameters are the same as in Figure \ref{fig2_un_toro}. Initial conditions of the neurons $x_{1,\beta},y_{1,\beta}$, $\beta=1,\ldots, 4$ of the torus $\#1$ are $x_{1,\beta}=(6.489,9.3843,6.9745,3.3656)$ $y_{1,\beta}=(0.8862,1.7536,8.3197,7.9935)$; initial conditions of the neurons $x_{11,\beta},y_{11,\beta}$, $\beta=8,\ldots, 11$ of the torus $\#2$ are $x_{1,\beta}=(0.5475,9.7797,2.3655,3.8290)$ $y_{1,j}=(4.8460,4.3027,7.5896,4.9924)$ while the initial conditions for the rest of neurons in both tori have  have been obtained by successive repetitions of the initial conditions used in Figure \ref{fig2_un_toro}. Power Spectrum Density showing that neurons of torus 2 (h) oscillate at $57.1=11\times 0.519$ Hz i.e. with 11--times the frequency of neurons from torus 1 (g).}
\label{two_tori}
\end{figure}

\section{Analysis of the dynamic behavior of one torus with $\mathbb{Z}_{N}\times\mathbb{Z}_{N}\times\mathbb{Z}_{2}$ symmetry}\label{onetorus_theor}

In a dissipative system such as (\ref{nFHN_2d}), one method for finding periodic solutions of a given type is Hopf bifurcation. This has been analyzed in the previous section, from the point of view of the eigenvalues obtained at the linearization about origin. In the following, we would like to analyze the periodic solutions to system (\ref{nFHN_2d}), from a group theoretic angle. A very useful tool is Theorem $3.4$ in \cite{GS04}.
\begin{priteo}\label{equiv}
    Let $\mathbf{\Gamma}$ be a finite group acting on $\mathbb{R}^{2n}.$ There is a periodic solution to some $\mathbf{\Gamma}$-equivariant system of ODEs on $\mathbb{R}^{2n}$ with spatial symmetries $K$ and spatio-temporal symmetries $H$ if and only if
            \begin{itemize}
                \item [(a)] H/K is cyclic;
                \item [(b)] K is an isotropy subgroup;
                \item [(c)] $\mathrm{dim}~\mathrm{Fix}(K)\geq2.$ If $\mathrm{dim}~\mathrm{Fix}(K)=2,$ then either $H=K$ or $H=N(K);$
                \item [(d)] H fixes a connected component of $\mathrm{Fix}(K)\setminus L_K.$
            \end{itemize}
\end{priteo}
When these conditions hold, there exists a smooth $\mathbf{\Gamma}$-equivariant vector field having a periodic solution with the desired symmetries. To find the kinds of periodic solutions in (\ref{nFHN_2d}), we need therefore to determine, up to conjugacy, all isotropy subgroups $K$ having $\mathrm{dim}~\mathrm{Fix}(K)\geq2$ and all subgroups $H$ for which $H/K$ is cyclic. In other words, all nontrivial periodic solutions of system (\ref{nFHN_2d}) correspond to the non trivial zeros in the two-dimensional fixed point space of some isotropy subgroup $\Sigma\subset \mathbb{Z}_{N}\times\mathbb{Z}_{N}\times\mathbb{Z}_{2}\times\mathbb{S}^{1}.$ In the process of identifying these isotropy subgroups, we will first reduce $\mathbf{\Gamma}\times\mathbb{S}^1$ to $\mathbb{Z}_N^{\alpha}\times\mathbb{Z}_N^{\beta}\times\mathbb{S}^1$ by considering a group homomorphism \cite{GS04}; then, we will show that the only isotropy subgroup $\mathbf{\Sigma}$ with $\mathrm{dim}~\mathrm{Fix}(\mathbf{\mathbf{\Sigma}})=2,$ is the kernel of $\mathbb{Z}_N^{\alpha}\times\mathbb{Z}_N^{\beta}\times\mathbb{S}^1.$ But before finding these two-dimensional isotropy subgroups, let us make the following considerations.

Recall first, that we have $\mathbf{\Gamma}=\mathbb{Z}_{N}^{\alpha}\times\mathbb{Z}_{N}^{\beta}\times\mathbb{Z}_{2}.$ If now we call $\mathbf{\Gamma}_1=\mathbb{Z}_{N}^{\alpha}\times\mathbb{Z}_{N}^{\beta},$ then we have $\mathbf{\Gamma}=\mathbf{\Gamma}_1\times\mathbb{Z}_{2}.$ For now, let focus our attention on the direct product group $\mathbf{\Gamma}_1.$ Let $\xi$ be a generator of $\mathbb{Z}_{N}.$ Its elements are written multiplicatively, in the form $\xi^\chi,$ and let $\omega=e^{2\pi i/N}.$ From $\mathbb{Z}_{N}^{\alpha}=\langle\sigma\rangle,$ we have $\sigma^N=1$ and from $\mathbb{Z}_{N}^{\beta}=\langle\rho\rangle,$ we have $\rho^N=1.$ By following the ideas presented in \cite{GS04}, the two-dimensional irreducible representations $\sigma_m$ and $\rho_m$ of $\mathbb{Z}_{N}^{\alpha}$ and $\mathbb{Z}_{N}^{\beta}$ respectively, over $\mathbb{C}$ are $\xi_0,\xi_1,\ldots,\xi_{[N/2]},$ defined as
\begin{equation}\label{ger1}
    \begin{array}{l}
        \xi z=\omega^mz,\hspace{0.7cm} m=1,2,\ldots,\frac{N-1}{2}.
    \end{array}
\end{equation}
$\xi_0$ is the trivial representation on $\mathbb{R}.$ When $N$ is even, $\xi_{N/2}$ is the representation on $\mathbb{R}$ in which both $\sigma$ and $\rho$ act trivially. In all other cases, $\xi_k$ (where $0\leq{k}\leq{N-1}$), is the representation on $\mathbb{R}^2=\mathbb{C}$ in which both $\sigma$ and $\rho$ act as multiplication by $\omega^k=e^{2\pi ik/N}.$

The space $\mathbb{R}^{N^2}=\langle{x}\rangle$ decomposes into $\mathbf{\Gamma}_1$ irreducibles according to $\mathbb{R}^{N^2}=x_0\oplus,\ldots,\oplus x_{[{N^2}/2]},$ where the action of $\mathbf{\Gamma}_1$ on $x_k$ is isomorphic to $\xi_k\times\xi_k$ and the action of $\varpi$ is by $-1.$ Similarly, the space $\mathbb{R}^{N^2}=\langle{y}\rangle$ decomposes into $\mathbf{\Gamma}_1$ irreducibles according to $\mathbb{R}^{N^2}=y_0\oplus,\ldots,\oplus y_{[{N^2}/2]},$ where the action of $\mathbf{\Gamma}_1$ on $y_k$ is isomorphic to $\xi_k\times\xi_k,$ and the action of $\varpi$ is by $-1.$ The $\mathbf{\Gamma}_1-$ irreducible components of $\mathbb{R}^{2N^2}$  are $x_k\oplus{y_k}$ with actions $(\xi_k\times\xi_k)\oplus(\xi_k\times\xi_k).$ Moreover, these are the isotypic components. The action of $\theta$ on each component is by $-1.$ As shown in \cite{GS04}, the action of $\mathbb{S}^1$ can be written in the form $x+iy=e^{i\theta}(x+iy).$ Therefore, $\pi\in\mathbb{S}^1$ also acts by $-1,$ so $(1,\varpi,\pi)\in\mathbf{\Gamma}_1\times\mathbb{Z}_2\times\mathbb{S}^1$ acts trivially.

There is a homomorphism $\mathbf{\Gamma}_1\times\mathbb{Z}_{2}\times\mathbb{S}^{1}\rightarrow\mathbf{\Gamma}_1\times\mathbb{S}^{1}$ defined by
\begin{equation*}\label{homom}
    \begin{array}{l}
        (\eta,1,\theta)\mapsto(\eta,\theta)\\
        (\eta,\varpi,\theta)\mapsto(\eta,\theta +\pi)
    \end{array}
\end{equation*}
and the action factors through this homomorphism, such that finally there is a $\mathbf{\Gamma}_1\times\mathbb{S}^{1}$ action, modulo $K=\langle(1,\varpi,\pi)\rangle.$ As shown in \cite{GS06}, the isotropy subgroups are generated by the isotropy subgroups of the $\mathbf{\Gamma}_1\times\mathbb{S}^{1}$ action together with $K.$

Since we know what the action of $\mathbf{\Gamma}_1$ on $\mathbb{C}$ is, the corresponding action of $\mathbf{\Gamma}_1\times\mathbb{S}^1$ is $(\xi,\xi,\theta)=\omega^{2m}e^{i\theta}z.$ The only way for $\mathbf{\Sigma}$ to have a two-dimensional fixed-point subspace is if $\mathbf{\Sigma}$ is the kernel of the $\mathbf{\Gamma}_1\times\mathbb{S}^1$ action. This kernel consists of the elements $(\xi^\chi,\xi^\chi,\theta)$ such that $\omega^{2m\chi}e^{i\theta}=1;$ that is, $4m\chi\pi/N+\theta=0.$ Hence $\mathbf{\Sigma}=\{(\xi^\chi,\xi^\chi,-4m\chi\pi/N)\}.$

This group is isomorphic to $\mathbf{\Gamma}_1.$ Its spatial part $\mathbf{\Sigma}\cap\mathbf{\Gamma}_1$ consists of those elements for which $m\chi$ is divisible by $N,$ so that $\chi$ is a multiple of $N/d$ where $d=gcd(m,N).$ That is, $\mathbf{\Sigma}\cap\mathbf{\Gamma}_1\cong\mathbf{\Gamma_d}.$ Therefore (\ref{ger1}) represents a discrete rotating wave with spatial symmetry $\mathbf{\Gamma_d}.$ The discrete rotating wave is translated diagonally in the squared lattice i.e. the discrete rotating wave corresponding to the lattice nodes $(\alpha,\alpha)$ with $0\leq{\alpha}\leq{N-\alpha},$ is the same. This diagonal displacement is determined by the simultaneous and orthogonal action of $\sigma$ and $\rho.$ As a consequence, there will exist $N$ different discrete rotating waves per lattice direction $\alpha$ or $\beta;$ the number of total different discrete rotating waves per lattice is still $N,$ because the discrete rotating waves in different rows (columns) of the lattice are just cyclic permutations of the $N-$ waves of one row (column).

Figure 3 shows the time evolution of the variables $x_{\alpha,\beta}$ for $\alpha,\beta={1,2,3},$ for the case $N=3$ of one single torus. It can be seen that for fixed $\alpha,$ $x_{\alpha,\beta}$ are three distinct discrete rotating waves, separated by a constant phase shift of $\phi=\frac{2\pi k}{N},$ as predicted by the above analysis. Moreover, when switching from $\alpha$ to $\alpha+1,$ variable $x_{\alpha+1,\beta}$ is again, phase shifted with $\frac{2\pi k}{N}$ with respect to $x_{\alpha,\beta}.$ In other words, discrete rotating waves are translated 'diagonally' in the lattice, so that we have ${H_{\alpha,\beta}}(t)=\left[{x_{\alpha,\beta}}(t-k\phi),{y_{\alpha,\beta}}(t-k\phi)\right]\cong{H_{\left(\alpha+k,\beta+k\right)\left(\mathrm{mod}\hspace{0.1cm} N\right)}}(t),$ with $0\leq{k}\leq{N-1}.$

\section{Multifrequencies in a network of two coupled tori, each with $\mathbf{\Gamma}_1$ symmetry}\label{twocoupled}
Multifrequencies patterns of oscillations have been initially observed in \cite{AC99}, \cite{GS02}, \cite{BGP00} and \cite{GS06}, where one or more oscillators oscillate at different frequencies. Golubitsky et al. and Stewart et al. developed the group-theoretical works which altogether represent the theoretical tools allowing a systematic analysis of these dynamics in systems with symmetry. Based on this theory, Palacios et al. \cite{PCLR05}, \cite{LPV03} as well as Longhini et al. \cite{LPV07}, developed arguments that explained the experimental observation of these patterns in electrical circuits. We apply here Golubitsky's theory and Palacios's methodology to study the multifrequency patterns in a network obtained by diffusively coupling two tori. Concretely, the
Equation (\ref{nFHN_2d}) becomes

\begin{equation}\label{nFHN_2d_2tori}
    \begin{array}{l}
        \displaystyle{\dot{x}_{{\left(\alpha,\beta\right)}_\psi}=ax_{{\left(\alpha,\beta\right)}_{\psi}}-
        x_{{\left(\alpha,\beta\right)}_{\psi}}^3-y_{{\left(\alpha,\beta\right)}_{\psi}}
        +\gamma \kappa\left(x_{{\left(\alpha+1,\beta\right)}_{\psi}},x_{{\left(\alpha,\beta\right)}_{\psi}}\right)}\\
        \\
        \hspace{1.63cm}\displaystyle{+\delta \mu\left(x_{{\left(\alpha,\beta+1\right)}_{\psi}},x_{{\left(\alpha,\beta\right)}_{\psi}}\right)
        +\varepsilon\frac{1}{N^2}\sum_{\varrho=1}^{N^2}{x_\varrho}_\upsilon}\\
        \\
        \displaystyle{\dot{y}_{{\left(\alpha,\beta\right)}_{\psi}}=
        bx_{{\left(\alpha,\beta\right)}_{\psi}}-cy_{{\left(\alpha,\beta\right)}_{\psi}}}
        \hspace{0.5cm}\psi,\upsilon=1,2\hspace{0.5cm}\psi\neq\upsilon.
    \end{array}
\end{equation}

Equation (\ref{nFHN_2d_2tori}) represents the dynamical system of the two coupled tori; indices $\psi,\upsilon,$ represent the two tori, $\varepsilon$ is the intertori coupling constant, while coupling functions $\kappa(x_{\alpha+1,\beta},x_{\alpha,\beta})$ and $\mu(x_{\alpha,\beta+1},x_{\alpha,\beta})$ are defined as in Equation (\ref{coupling}).

Let us first observe that while each torus has interneuronal nearest-neighbor coupling in each direction, the intertori coupling is such that every neuron in one torus is coupled with a 'mean field' of the $N^2$ neurons of the other torus. Therefore, while the symmetry group acting on each torus is $\mathbb{Z}_{N}^{\alpha}\times\mathbb{Z}_{N}^{\beta}\times\mathbb{Z}_{2},$ the symmetry group of the full system is given by the wreath product $\left(\mathbb{Z}_{N}^{\alpha}\times\mathbb{Z}_{N}^{\beta}\times\mathbb{Z}_{2}\right)^2\wr\mathbb{D}_2.$ A formalism has been developed in \cite{DGS04}, which allows determining all oscillation patterns of the full system \eqref{nFHN_2d_2tori}. The methodology consists in calculating the axial subgroups (up to conjugacy) of the wreath product $\left(\mathbb{Z}_{N}^{\alpha}\times\mathbb{Z}_{N}^{\beta}\times\mathbb{Z}_{2}\right)^2\wr\mathbb{D}_2$ and analyzing all isotropy subgroups and maximal isotropy subgroups. It is a general approach able to theoretically predict the patterns of oscillation of a network formed by the nearest neighbor coupling of an arbitrary number of tori. This makes the object of another paper and is actually our work in progress that it will be published elsewhere shortly. In this Section, however, we do not pretend to carry out such a global analysis. The purpose of this section is to prove, by using a group theoretical argument, the possibility of a pattern of oscillations in which one torus produces discrete rotating waves with a constant phase shift, while the other torus shows in-phase oscillations at $N-$times the frequency of the rotating waves. We observed this multifrequency oscillation pattern by computer simulations of system (\ref{nFHN_2d_2tori}) and would like to theoretically prove the possibility of obtaining it.

The idea is constructing a set of multifrequency patterns of interest and analyzing the action the symmetry group of the network has on the constructed set. It should be emphasized, the constructed set is only one of the possible sets of patterns of oscillations of the network.

Let us first define
\begin{equation*}\label{set}
    \begin{array}{l}
        {H_1}(t)=\left[x_{\alpha,\beta}(t),y_{\alpha,\beta}(t)\right]_{T_1}\\
        {H_2}(t)=\left[x_{\alpha,\beta}(t),y_{\alpha,\beta}(t)\right]_{T_2}
    \end{array}
\end{equation*}
to represent the state of the $x_{\alpha,\beta},y_{\alpha,\beta}$ variables in one tori $\#1$ and $\#2,$ respectively. Then the set $P(t)=[H_1(t),H_2(t)],$ represents an ensemble of the spatio-temporal patterns generated by the network. Computer simulations of system (\ref{nFHN_2d}) corresponding to two coupled tori indicate that a possible set of patterns of oscillations is formed by discrete rotating waves with a constant phase shift in one torus and in-phase oscillations generated by the other torus. Figures \ref{two_tori}(a) and \ref{two_tori}(e), show three distinct traveling waves in two coupled tori of $3\times 3$ neurons (Figure \ref{two_tori}(a)) and $11\times 11$ neurons (Figure \ref{two_tori}(e)). The waves generated by torus $\#1,$ are of the form
\begin{equation*}\label{TW}
    \begin{array}{l}
        {H_1}_{TW}(t)=\left[{x_{\alpha,\beta}}_{TW}\left(t-\left(\alpha+\beta\right)\left(\mathrm{mod}\hspace{0.1cm} N\right)\phi\right),{y_{\alpha,\beta}}_{TW}\left(t-\left(\alpha+\beta\right)\left(\mathrm{mod}\hspace{0.1cm} N\right)\phi\right)\right]
    \end{array}
\end{equation*}
with a constant phase shift $\phi=\frac{2\pi k}{N}=\frac{\tau}{N},$ -where $\tau$ is the period of oscillations- among nearest-neighbor oscillators, the waves being cyclically permuted within the rows (columns) of the squared array. On the other hand, torus $\#2$ generates in-phase oscillations of the form
\begin{equation*}\label{IP}
    \begin{array}{l}
        {H_2}_{IP}(t)=\left[{x_{\alpha,\beta}}_{IP}(t),{y_{\alpha,\beta}}_{IP}(t)\right]
    \end{array}
\end{equation*}
of same period $\tau,$ where all variables oscillate at $N-$ times the frequency of the traveling waves and all $[{x_{\alpha,\beta}}_{IP}(t),{y_{\alpha,\beta}}_{IP}(t)]$ are identical. A sample of the in-phase oscillations in two coupled tori of $3\times 3$ neurons is shown in Figure \ref{two_tori}(b) and for $11\times 11$ neurons, in Figure \ref{two_tori}(f). Therefore in this case, the set of patterns of oscillations ${P_{2T}}(t)=\left[{H_1}_{TW}(t),{H_2}_{IP}(t)\right],$ is formed by traveling waves ($TW$) and in-phase ($IP$) oscillations (where $2T$ stands for the case of two coupled tori).

Now, it appears natural to search the spatial and temporal transformations that leave unchanged the elements of the set of patterns of oscillations. The set of these spatial and temporal transformations form the symmetry group of the oscillation patterns \cite{LPV03}, \cite{LPV07}. As a first step, let's assume that the pattern $P_{2T}=\left[{H_1}_{TW}(t),{H_2}_{IP}(t)\right]$ has the symmetry group $\left(\mathbf{\Gamma}_1\times\mathbb{S}^1\right)\times \left(\mathbf{\Gamma}_1\times\mathbb{S}^1\right),$ which describes (simultaneous) cyclic permutations of the oscillators in each array, accompanied by shifts in time by $\phi.$ That is, $\left(\mathbf{\Gamma}_1\times\mathbb{S}^1\right)\cdot {H_1}_{TW}(t)={H_1}_{TW}(t)$ so the traveling waves are unchanged. On the other hand, $\left(\mathbf{\Gamma}_1\times\mathbb{S}^1\right)\cdot {H_2}_{IP}(t)={H_2}_{IP}(t+\phi),$ so the in-phase oscillators are shifted in time by $\phi.$ Thus if $\left(\mathbf{\Gamma}_1\times\mathbb{S}^1\right)\times \left(\mathbf{\Gamma}_1\times\mathbb{S}^1\right)$ is the symmetry group of $P_{2T}(t),$ then $\left(\mathbf{\Gamma}_1\times\mathbb{S}^1\right)\times \left(\mathbf{\Gamma}_1\times\mathbb{S}^1\right)\cdot P_{2T}(t)=P_{2T}(t),$ implies that ${H_2}_{IP}(t)={H_2}_{IP}(t+\phi).$ There is a simple explanation for the fast oscillations generated by the second torus: the only way to obtain in-phase oscillations in torus $\#2$ once we get traveling waves with constant phase shift in torus $\#1,$ is when in-phase pattern oscillates at $N$ times the frequency of the traveling wave pattern \cite{GS04}.

When $N$ is even, however, a condition for the Hopf bifurcation of the discrete rotating waves, taking into account is not only nearest-neighbor but also next-nearest neighbor couplings \cite{GS86},\cite{LPV03}. As shown in \cite{AC99}, there is a more subtle interpretation for bifurcation of discrete rotating waves when $m$ and $N$ are not coprime.

As it was shown in Section \ref{onetorus_theor} and \cite{AC99}, the spatial part $\mathbf{\Sigma}\cap\mathbf{\Gamma}_1$ consists of those elements for which $m\chi$ is divisible by $N,$ so that $\chi$ is a multiple of $\iota=N/d$ where $d=gcd(m,N).$ Then the ring of oscillators can be divided into $\iota$ subsets of $d$ oscillators, so that in each subset the oscillators behave identically (they oscillate in-phase), but the next subset oscillates with a phase-lag of $T/\iota$ with respect to the first one. We therefore get a discrete rotating wave among sets of oscillators behaving identically.

\section{Concluding Remarks}

We analyzed the oscillatory dynamics of a network of electrically coupled FHN neurons. This network can be related to architectures of electrically coupled neurons observed in anatomical structures of the nervous system. We first described the building-block of the network, which consists of squared arrays shaped in the form of a torus. The analysis of the linear approximation of the system formed by one torus reveals the importance of interneuronal coupling strengths; they act as precursors of the Hopf bifurcation point. We give the analytical conditions for Hopf bifurcation and asymptotic stability of the coupled system, in which each torus is formed by an arbitrary number of neurons. We then performed a group theoretical analysis of one torus and found that only oscillatory pattern is represented by traveling waves with a constant phase shift. When analyzing two coupled tori, we identified a set of patterns of oscillations whose elements are traveling waves and in-phase 'fast' oscillations. We showed that when the symmetry group of a single torus acts on the traveling waves, it leaves them unchanged; by the contrary, when it acts on the in-phase oscillations, they are shifted in time by $\phi.$ This proves that one of the possible patterns of oscillations of the network formed by the two coupled tori is represented by traveling waves produced by one torus and in-phase oscillations at $N-$times the frequency of the traveling waves, shown by the other torus. This result is of possible interest for modeling the electrical activity of the nervous system.

\textbf{Acknowledgments.}
ACM acknowledges support from the BioSim Network, grant number LSHB-CT-2004-005137. Thanks are expressed to Prof. Peter Ashwin for helpful discussions. 


\appendix
\section{Appendix}\label{appendix}

Our aim is applying the criterium established by Guckenheimer and Holmes \cite{guck}, to analyze the stability of the limit cycle of system \eqref{FHN1} obtained by linearization about origin. From equation \eqref{eigvalsFHN1}, the condition for Hopf bifurcation is $a=c<b,~a<1,$ and system \eqref{FHN1} writes

\begin{equation}\label{FHN1hopf}
    \begin{array}{l}
        \dot{x}=ax-x^3-y\\
        \dot{y}=bx-ay.
    \end{array}
\end{equation}

It has been shown in \cite{guck}, \cite{strog}, that any system at Hopf bifurcation can be put in the form

\begin{equation}\label{sist-strog}
    \begin{array}{l}
        \dot{x}=-\varphi y +f\left(x,y\right)\\
        \dot{y}=\varphi x +g\left(x,y\right),
    \end{array}
\end{equation}

by a suitable change of variables. In equation \eqref{sist-strog}, $\varphi$ is a real number, while $f\left(x,y\right)$ and $g\left(x,y\right)$ contain only higher-order nonlinear terms that vanish at origin. Moreover, one can discern between a subcritical and supercritical Hopf bifurcation by calculating the sign of the quantity:

\begin{equation}\label{derpar}
\begin{array}{l}
        \displaystyle{16s^*=f_{xxx}+f_{xyy}+g_{xxy}+g_{yyy}}\\
        \\
        \hspace{1.1cm}\displaystyle{+\frac{1}{\varphi}\left[f_{xy}\left(f_{xx}+f_{yy}\right)-g_{xy}\left(g_{xx}+g_{yy}\right)-f_{xx}g_{xx}+f_{yy}g_{yy}\right]},
\end{array}
\end{equation}

where the subscripts indicate partial derivatives evaluated at $\left(0,0\right).$ If $s^*<0,$ the bifurcation is supercritical, while if $s^*>0$ the bifurcation is subcritical.

In order to put system \eqref{FHN1hopf} in the form \eqref{sist-strog}, we make the transformation $\displaystyle{\tilde{x}=\frac{1}{\sqrt{b-a^2}}y-\frac{a}{\sqrt{b-a^2}}x,~\tilde{y}=x.}$ System \eqref{FHN1hopf} is transformed into

\begin{equation*}\label{FHN1cambiado}
    \begin{array}{l}
        \displaystyle{\dot{\tilde{x}}=-\varphi \tilde{y} +\frac{a}{\sqrt{b-a^2}}\tilde{y}^3}\\
        \displaystyle{\dot{\tilde{y}}= \varphi\tilde{x}-\tilde{y}^3},
    \end{array}
\end{equation*}

where $\displaystyle{\varphi=-\sqrt{b-a^2},~f\left(\tilde{x},\tilde{y}\right)=\frac{a}{\sqrt{b-a^2}}\tilde{y}^3,~
g\left(\tilde{x},\tilde{y}\right)=-\tilde{y}^3}.$ Then, by evaluating expression \eqref{derpar}, we obtain $16s^*=-6$, which gives $s^*<0$ so Hopf bifurcation corresponding to steady state $\left(0,0\right),$ is supercritical.


\begin{thebibliography}{00}


\bibitem{H06}{\sc R. B. Hoyle}, {\it Pattern formation: an introduction to methods}, Cambridge Univ. Press, (2006), 116--122.
\bibitem{GS04b}{\sc M. Golubitsky, J. Stewart}, {\it The symmetry perspective: from equilibrium to chaos in phase space and physical space}, Birkhauser, (2004), 91--95.
\bibitem{BBKS00}{\sc G.N. Borisyuk, , R.M. Borisyuk, Y.B.Kazanovich, G. Strong}, {\it Oscillations in neural systems}, Lawrence Erlbaum Associates, Publishers, (2000), 261-284.
\bibitem{BBK98}{\sc R.M. Borisyuk, G.N. Borisyuk, Y.B.Kazanovich}, {\it The synchronization principle in modelling of binding and attention}, Membr. Cell Biol., \textbf{11}, (1998), 753--761.
\bibitem{PCLR05}{\sc A. Palacios, R. Carretero-González, P. Longhini, N. Renz}, {\it Multifrequency synthesis using two coupled nonlinear oscillator arrays}, Phys. Rev. E, \textbf{72}, (2005), 026211--9.
\bibitem{BGP00}{\sc P.L. Buono, M. Golubitsky and A. Palacios}, {\it Heteroclinic cycles in rings of coupled cells}, Physica D, \textbf{143}, (2000), 74--108.
\bibitem{GSB98}{\sc M. Golubitsky, I. Stewart, P.L. Buono and J.J. Collins}, {\it A modular network for legged locomotion}, Physica D, \textbf{115}, (1998), 56--72.
\bibitem{KE88}{\sc N. Kopell, G.B. Ermentrout}, {\it Coupled oscillators and the design of central pattern generators}, Math. Biosci., \textbf{90}, (1988), 87--109.
\bibitem{KE90}{\sc N. Kopell, G.B. Ermentrout}, {\it Phase transitions and other phenomena in chains of coupled oscillators}, SIAM J. App. Math., \textbf{50}, (1990), 1014--1052.
\bibitem{RTWE98}{\sc J. Rinzel, D. Terman, X.J. Wang, B. Ermentrout}, {\it Propagating activity patterns in large-scale inhibitory neuronal networks}, Science, \textbf{279}, (1998), 1351--1355.
\bibitem{CS94}{\sc J.J. Collins, I. Stewart}, {\it A group-theoretic approach to rings of coupled biological oscillators}, Biol. Cyb., \textbf{71}, (1994), 95--103.

\bibitem{AC99}{\sc D. Armbruster, P. Chossat}, {\it Remarks on multi-frequency oscillations in symmetrically coupled oscillators}, Phys. Lett. A, \textbf{254}, (1999), 269--274.

\bibitem{GS86}{\sc M. Golubitsky, I. Stewart}, {\it Hopf bifurcation with dihedral group symmetry: coupled nonlinear oscillators. In: Multiparameter bifurcation theory}, M. Golubitsky, J. Guckenheimer, eds., Contemporary Mathematics \textbf{56}, AMS (1986), 131--173.
\bibitem{GS04}{\sc M. Golubitsky, J. Stewart}, {\it Singularities and groups in bifurcation theory II}, M. Golubitsky, I. Stewart, D. G. Schaeffer, eds., Applied mathematical sciences \textbf{69}, Springer-Verlag, (1988), 388--399.
\bibitem{MDHM00}{\sc J. Miller, W.P. Dayawansa, P. Hallgren, C.F. Martin}, {\it Phase Locking in the mammalian circadian clock}, Proc. IEEE Conf. on Decision and Control,(2000), 1643--1648.
\bibitem{LPV07}{\sc P. Longhini, A. Palacios, V. In, J.D. Neff, A. Kho, A. Bulsara}, {\it Exploiting dynamical symmetry in coupled nonlinear elements for efficient frequency down-conversion}, Phys. Rev. E, \textbf{76}, (2007), 026201--6.
\bibitem{LPV03}{\sc V. In, A. Kho, J.D. Neff, A. Palacios, P. Longhini, B.K. Meadows}, {\it Experimental observation of multifrequency patterns in arrays of coupled nonlinear oscillators}, Phys. Rev. Lett., \textbf{91}, (2003), 244101-4.
\bibitem{SS07}{\sc G.B. Stan, R. Sepulchre}, {\it Analysis of interconnected oscillators by dissipativity theory}, IEEE Trans. Autom. Control, \textbf{52}, (2007), 256--270.
\bibitem{GS02}{\sc M. Golubitsky, J. Stewart}, {\it Geometry, dynamics, and mechanics: 60th Birthday Volume for J. E. Marsden, edited by P. Holmes, P. Newton, and A. Weinstein}, Springer-Verlag, Berlin, (2002), 243--286.
\bibitem{M02}{\sc J.D. Murray}, {\it Mathematical biology, I: an introduction}, Springer-Verlag, (2002), 239--244.
\bibitem{GS06}{\sc M. Golubitsky, J. Stewart}, {\it Nonlinear dynamics of networks: the group formalism}, Bull. AMS, \textbf{43}, (2006), 305--364.
\bibitem{T97}{\sc R. Turcajova}, {\it Numerical condition of discrete wavelet transforms}, SIAM J. Matrix Anal. Appl., \textbf{18}, (1997), 981--999.

\bibitem{Cortex}{\sc A.K. Seth, J.L. McKinstry, G.M. Edelman, J.L. Krichmar}, {\it Visual binding through reentrant connectivity and dynamic synchronization in a brain-based device}, Cer. Cortex, \textbf{14}, (2004), 1185--1199.
\bibitem{DGS04}{\sc B. Dionne, M. Golubitsky, I. Stewart}, {\it Coupled cells with internal symmetry: I. Wreath products}, Nonlinearity, \textbf{9}, (2004), 559--574.
\bibitem{Z}{\sc E.C. Zeeman}, {\it Catastrophe theory in brain modelling}, Int. J. Neuroscience, \textbf{6}, (1973), 39--41.
\bibitem{HH}{\sc A. Hodgkin, A. Huxley}, {\it A quantitative description of membrane current and its application to conduction and excitation in nerve}, J. Physiol., \textbf{117}, (1952), 500--544.
\bibitem{FN1}{\sc J.S. Nagumo, S. Arimoto, S. Yoshizawa}, {\it An active pulse transmission line simulating nerve axon}, Proc. IRE, \textbf{50}, (1962), 2061--2071.
\bibitem{FN2}{\sc R. FitzHugh}, {\it Impulses and physiological states in theoretical models of nerve membrane}, Biophys. J., \textbf{1}, (1961), 445--466.
\bibitem{Phys}{\sc A. Takamatsu, R. Tanaka, H. Yamada, T. Nakagaki, T. Fuji, I. Endo}, {\it Spatiotemporal symmetry in rings of coupled biological oscillators of Physarum Plasmodial slime mold}, Phys. Rev. Lett., \textbf{87}, (2001), 078102--4.
\bibitem{Cat}{\sc M.S. Livingstone,  D.H. Hubel}, {\it Specificity of intrinsic connections in primate primary visual cortex}, J. Neurosci., \textbf{4}, (1984), 2830--2835.
\bibitem{VisGol}{\sc P.C. Bressloff, J.D. Cowan, M. Golubitsky, P.J. Thomas, M.C. Wiener}, {\it Geometric visual hallucinations, Euclidean symmetry, and the functional architecture of striate cortex}, Phil. Trans. Roy. Soc. (Lond.) B, \textbf{356}, (2001), 1--32.
\bibitem{martinerie}{\sc F. Varela, J.P. Lachaux, E. Rodriguez, J. Martinerie}, {\it The brainweb: phase synchronization and large-scale integration}, Nature, \textbf{2}, (2001), 229--239.
\bibitem{coexistence}{\sc A.O. Komendantov, C.C. Canavier}, {\it Electrical coupling between model midbrain dopamine neurons: effects on firing pattern and synchrony}, Nature, \textbf{87}, (2002), 1526--1541.
\bibitem{learning}{\sc N. Kopell, G.B. Ermentrout}, {\it Learning of phase lags in coupled neural oscillators}, Neur. Comp., \textbf{6}, (1994), 225--241.
\bibitem{guck}{\sc J. Guckenheimer, P. Holmes}, {\it Nonlinear oscillations, dynamical systems, and bifurcation of vector fields}, J. Marsden, L. Sirovich, F. John, eds., Applied mathematical sciences \textbf{42}, Springer-Verlag, (1983), 152--156.
\bibitem{strog}{\sc S.H. Strogatz}, {\it Nonlinear dynamics and chaos: with applications to physics, biology, chemistry and engineering}, Westview Press, (1994), 289.
\bibitem{kostova}{\sc T. Kostova, R. Ravindran, M. Schonbek}, {\it FitzHugh-Nagumo revisited: types of bifurcations, periodical forcing}, Int. J. Bif. Chaos, \textbf{14}, (2004), 913--925.

\end{thebibliography}
\end{document}